\documentclass[12pt]{article}
\usepackage{e-jc}

\usepackage{amsthm,amsmath,amssymb}

\usepackage{graphicx}

\usepackage[colorlinks=true,citecolor=black,linkcolor=black,urlcolor=blue]{hyperref}

\theoremstyle{plain}
\newtheorem{theorem}{Theorem}
\newtheorem{lemma}[theorem]{Lemma}
\newtheorem{corollary}[theorem]{Corollary}
\newtheorem{proposition}[theorem]{Proposition}

\theoremstyle{definition}
\newtheorem{definition}[theorem]{Definition}
\newtheorem{example}[theorem]{Example}
\newtheorem{conjecture}[theorem]{Conjecture}

\theoremstyle{remark}
\newtheorem{remark}[theorem]{Remark}

\title{\bf On periodicity of generalized pseudostandard words}

\author{Josef Florian\thanks{supported by the grant SGS14/205/OHK4/3T/14}\\
\small Department of Mathematics\\[-0.8ex]
\small Faculty of Nuclear Sciences and Physical Engineering\\[-0.8ex]
\small Czech Technical University in Prague\\
\small\tt florijos@fjfi.cvut.cz\\
\and
L\!'ubom\'ira Dvo\v r\'akov\'a (born Balkov\'a) \thanks{supported by Czech Science Foundation
grant GA\v CR 13-03538S}\\
\small Department of Mathematics\\[-0.8ex]
\small Faculty of Nuclear Sciences and Physical Engineering\\[-0.8ex]
\small Czech Technical University in Prague\\
\small\tt lubomira.balkova@fjfi.cvut.cz
}

\date{\dateline{Jan 1, 2012}{Jan 2, 2012}}

\begin{document}

\maketitle


\begin{abstract}
  Generalized pseudostandard words were introduced by de Luca and De Luca in 2006~\cite{LuDeLu}. In comparison to the palindromic and pseudopalindromic closure, only little is known about the generalized pseudopalindromic closure and the associated generalized pseudostandard words. In this paper we provide a necessary and sufficient condition for their periodicity over binary and ternary alphabet. More precisely, we describe how the directive bi-sequence of a~generalized pseudostandard word has to look like in order to correspond to a~periodic word. We state moreover a~conjecture concerning a~necessary and sufficient condition for periodicity over any alphabet.

\end{abstract}

\section{Introduction}
This paper focuses on a recent topic of combinatorics on words: generalized pseudostandard words. Such words were defined by de Luca and De Luca in 2006~\cite{LuDeLu} and generalize standard episturmian words, resp.,  pseudostandard words -- instead of the palindromic closure, resp., one pseudopalindromic closure, an infinite sequence of involutory antimorphisms is considered.
While standard episturmian and pseudostandard words have been studied intensively and a~lot of their properties are known (see for instance~\cite{BuLuZa,DrJuPi,Lu,LuDeLu}), only little has been shown so far about the generalized pseudopalindromic closure that gives rise to generalized pseudostandard words. In the paper~\cite{LuDeLu} the authors have defined the generalized pseudostandard words and have proven there that the famous Thue--Morse word is an example of such words. Jajcayov\'a et al.~\cite{JaPeSt} have characterized generalized pseudostandard words in the class of generalized Thue--Morse words.
Jamet et al.~\cite{JaPaRiVu} deal with fixed points of the palindromic and pseudopalindromic closure and formulate an open problem concerning fixed points of the generalized pseudopalindromic closure. The most detailed study of generalized pseudostandard words has been so far provided by Blondin Mass\'e et al.~\cite{MaPa}:
 \begin{itemize}
 \item An algorithm for normalization over binary alphabet is described. This algorithm transforms the directive bi-sequence in such a~way that the obtained word remains unchanged and no pseudopalindromic prefix is missed during the construction.
 \item An effective algorithm -- the generalized Justin's formula -- for producing generalized pseudostandard words is presented.
 \item The standard Rote words are proven to be generalized pseudostandard words and the infinite sequence of antimorphisms that generates such words is studied.
\end{itemize}

In this paper we provide a sufficient and necessary condition for periodicity of generalized pseudostandard words over binary and ternary alphabet. More precisely, we describe how the directive bi-sequence of a~generalized pseudostandard word has to look like in order to correspond to a~periodic word.

The text is organized as follows. In Section~\ref{sec:CoW} we introduce basics from combinatorics on words.
In Section~\ref{sec:generalized_pseudopalindrome}, the generalized pseudopalindromic closure is defined, the normalization algorithm over binary alphabet from~\cite{MaPa} is recalled and some new partial results on normalization over ternary alphabet are provided. The main results are presented in the following two sections. A~sufficient and necessary condition for periodicity of generalized pseudostandard words is given in Section~\ref{sec:per} over binary alphabet and in Section~\ref{sec:perTernary} over ternary alphabet. Finally, a~conjecture concerning a~necessary and sufficient condition for periodicity over any alphabet is stated in the last section.

\section{Basics from combinatorics on words}\label{sec:CoW}
Any finite set of symbols is called an \emph{alphabet} $\mathcal A$, the elements are called \emph{letters}. A~\emph{(finite) word} $w$ over $\mathcal A$ is any finite sequence of letters. Its length $|w|$ is the number of letters it contains. The \emph{empty word} -- the neutral element for concatenation of words -- is denoted $\varepsilon$ and its length is set $|\varepsilon|=0$. The symbol ${\mathcal A}^*$ stands for the set of all finite words over $\mathcal A$.
An \emph{infinite word} $\mathbf u$ over $\mathcal A$ is any infinite sequence of letters.
A finite word $w$ is a~\emph{factor} of the infinite word $\mathbf u=u_0u_1u_2\ldots$ with $u_i \in \mathcal A$ if there exists an index $i\geq 0$ such that $w=u_iu_{i+1}\ldots u_{i+|w|-1}$. Let $u,v,w \in {\mathcal A}^*$, then for $w=uv$ we mean by $wv^{-1}$ the word $u$ and by $u^{-1}w$ the word $v$. The symbol ${\mathcal L}(\mathbf u)$ is used for the set of factors of $\mathbf u$ and is called the \emph{language} of $\mathbf u$, similarly ${\mathcal L}_n(\mathbf u)$ stands for the set of factors of $\mathbf u$ of length $n$.


\noindent Let $w \in \mathcal{L}({\mathbf u})$. A~\emph{left extension} of $w$ is any word $aw \in \mathcal{L}(\mathbf{u})$, where $a \in \mathcal A$. The factor $w$ is called \emph{left special} if $w$ has at least two left extensions.
The \emph{(factor) complexity} of $\mathbf{u}$ is the map $\mathcal{C}_{{\mathbf u}}: {\mathbb N} \rightarrow {\mathbb N}$ defined as $$\mathcal{C}_{{\mathbf u}}(n) =\#\mathcal{L}_n(\mathbf u).$$
The following results on complexity come from~\cite{MoHe}.
If an infinite word is eventually periodic, i.e., it is of the form $wv^{\omega}$, where $w,v$ are finite words ($w$ may be empty -- in such a case we speak about a~purely periodic word) and $\omega$ denotes an infinite repetition, then its factor complexity is bounded. An infinite word is not eventually periodic -- such a~word is called aperiodic -- if and only if its complexity satisfies: ${\mathcal C}(n)\geq n+1$ for all $n \in \mathbb N$. If an infinite word $\mathbf u$ contains for every length $n$ a~left special factor of length $n$, the complexity is evidently strictly growing, hence $\mathbf u$ is aperiodic.
An infinite word $\mathbf u$ is called \emph{recurrent} if each of its factors occurs infinitely many times in $\mathbf u$. It is said to be \emph{uniformly recurrent} if for every $n \in \mathbb N$ there exists a~length $r(n)$ such that every factor of length $r(n)$ of $\mathbf u$ contains all factors of length $n$ of $\mathbf u$.


An \emph{involutory antimorphism} is a~map $\vartheta: {\mathcal A}^* \rightarrow {\mathcal A}^*$ such that for every $v, w \in {\mathcal A}^*$ we have $\vartheta(vw) = \vartheta(w) \vartheta(v)$ and moreover $\vartheta^2$ equals identity. It is clear that in order to define an antimorphism, it suffices to provide letter images. There are only two involutory antimorphisms over the alphabet $\{0,1\}$: the \emph{reversal (mirror) map} $R$ satisfying $R(0)=0,\ R(1)=1$, and the \emph{exchange antimorphism} $E$ given by $E(0)=1,\ E(1)=0$.
We use the notation $\overline{0} = 1$ and $\overline{1} = 0$, $\overline{E} = R$ and $\overline{R} = E$.
There are only four involutory antimorphisms over the alphabet $\{0, 1, 2\}$: the reversal map $R$ satisfying $R(0)=0,\ R(1)=1,\ R(2)=2$, and three exchange antimorphisms $E_0,\ E_1,\ E_2$ given by $$\begin{array}{lll}
E_0(0)=0,&\ E_0(1)=2,&\ E_0(2)=1\\
E_1(0)=2,&\ E_1(1)=1,&\ E_1(2)=0\\
E_2(0)=1,&\ E_2(1)=0,&\ E_2(2)=2\; .
 \end{array}$$

Consider an involutory antimorphism $\vartheta$ over $\mathcal A$. A finite word $w$ is a~$\vartheta$-\emph{palindrome} if $w = \vartheta(w)$. The $\vartheta$-\emph{palindromic closure} $w^{\vartheta}$ of a~word $w$ is the shortest $\vartheta$-palindrome having $w$ as prefix. For instance, over binary alphabet $011^R=0110,\ 011^{E}=011001$.
We speak about the \emph{palindromic closure} if $\vartheta=R$ and the \emph{pseudopalindromic closure} if we do not need to specify which antimorphism $\vartheta$ is used.

\section{Generalized pseudopalindromic closure}\label{sec:generalized_pseudopalindrome}
Generalized pseudostandard words form a~generalization of infinite words generated by the palindromic or pseudopalindromic closure (see~\cite{BuLuZa,DrJuPi,Lu,LuDeLu} for more details on pseudostandard words); such a~construction was first described and studied in~\cite{LuDeLu}.
Let us start with their definition and known properties; we use the papers~\cite{JaPaRiVu,LuDeLu,MaPa}.

\subsection{Definition of generalized pseudostandard words}
%
%

\begin{definition}
Let $\mathcal A$ be an alphabet and $G$ be the set of all involutory antimorphisms on ${\mathcal A}^*$.
Let $\Delta = \delta_1 \delta_2 \ldots$ and $\Theta = \vartheta_1 \vartheta_2 \ldots$, where $\delta_i \in {\mathcal A}$ and $\vartheta_i \in G$ for all $i \in \mathbb{N}$. The infinite \emph{generalized pseudostandard word} $\mathbf{u}(\Delta, \Theta)$ generated by the generalized pseudopalindromic closure is the word whose prefixes $w_n$ are obtained from the recurrence relation
$$w_{n+1} = (w_n \delta_{n+1})^{\vartheta_{n+1}},$$ $$w_0 = \varepsilon.$$ The sequence $\Lambda = (\Delta, \Theta)$ is called the \emph{directive bi-sequence} of the word $\mathbf{u}(\Delta, \Theta)$.
\end{definition}
If $\Theta = \vartheta^{\omega}$ in the previous definition, then we deal with well known \emph{pseudostandard words}. In such a~case the sequence $(w_n)_{n \geq 0}$ is known to contain all $\vartheta$-palindromic prefixes of $\mathbf{u}(\Delta, \Theta)$.

We will restrict our considerations to two cases:
\begin{enumerate}
\item binary alphabet ${\mathcal A}=\{0,1\}$: $G=\{R, E\}$,
\item ternary alphabet ${\mathcal A}=\{0,1,2\}$: $G=\{R, E_0, E_1, E_2\}$, where $E_0, E_1, E_2$ have been defined in Section~\ref{sec:CoW}.
\end{enumerate}

The following two properties are readily seen from the definition of $\mathbf u=\mathbf{u}(\Delta, \Theta)$:
 \begin{itemize}
 \item If an involutory antimorphism $\vartheta$ is contained in $\Theta$ infinitely many times, then the language of $\mathbf u$ is closed under the antimorphism $\vartheta$.
\item The word $\mathbf u$ is uniformly recurrent.
\end{itemize}

\subsection{Normalization over binary alphabet}
In contrast to pseudostandard words, the sequence $(w_n)_{n\geq 0}$ of prefixes of a~binary generalized pseudostandard word ${\mathbf u}(\Delta, \Theta)$ does not have to contain all palindromic and $E$-palindromic prefixes of ${\mathbf u}(\Delta, \Theta)$. Blondin Mass\'e et al.~\cite{MaPa} have introduced the notion of normalization of the directive bi-sequence.

\begin{definition}\label{def:norm}
A~directive bi-sequence $\Lambda=(\Delta, \Theta)$ of a~binary generalized pseudostandard word $\mathbf{u}(\Delta, \Theta)$ is called \emph{normalized} if the sequence of prefixes $(w_n)_{n\geq 0}$ of $\mathbf{u}(\Delta, \Theta)$ contains all palindromic and $E$-palindromic prefixes of ${\mathbf u}(\Delta, \Theta)$.
\end{definition}

\begin{example} \label{ex:norm}
Let $\Lambda=(\Delta, \Theta) = ((011)^{\omega}, (EER)^{\omega})$. Let us write down the first prefixes of $\mathbf{u}(\Delta, \Theta)$:
\begin{align}
w_1 =& \;01 \nonumber \\
w_2 =& \;011001 \nonumber \\
w_3 =& \;01100110 \nonumber \\
w_4 =& \;0110011001. \nonumber
\end{align}
The sequence $w_n$ does not contain for instance palindromic prefixes $0$ and $0110$ of $\mathbf{u}(\Delta, \Theta)$.
\end{example}
The authors of~\cite{MaPa} have proven that every directive bi-sequence $\Lambda$ can be normalized, i.e., transformed to such a~form $\widetilde \Lambda$ that the new sequence $(\widetilde{w}_n)_{n\geq 0}$ contains already all palindromic and $E$-palindromic prefixes and $\widetilde \Lambda$ generates the same generalized pseudostandard word as~$\Lambda$.

\begin{theorem}\label{thm:norm}
Let $\Lambda = (\Delta, \Theta)$ be a~directive bi-sequence of a~binary generalized pseudostandard word. Then there exists a~normalized directive bi-sequence $\widetilde{\Lambda} = (\widetilde{\Delta}, \widetilde{\Theta})$ such that ${\mathbf u}(\Delta, \Theta) = {\mathbf u}(\widetilde{\Delta}, \widetilde{\Theta})$.

Moreover, in order to normalize the sequence $\Lambda$, it suffices firstly to execute the following changes  to its prefixes (if they are of the corresponding form):
\begin{itemize}
\item $(a\bar{a}, RR) \rightarrow (a\bar{a}a, RER)$,
\item $(a^i, R^{i-1}E) \rightarrow (a^i\bar{a}, R^iE)$ for $i \geq 1$,
\item $(a^i\bar{a}\bar{a}, R^iEE) \rightarrow (a^i\bar{a}\bar{a}a, R^iERE)$ for $i \geq 1$,
\end{itemize}
and secondly to replace step by step from left to right every factor of the form:
\begin{itemize}
\item $(ab\bar{b}, \vartheta\overline{\vartheta}\overline{\vartheta}) \rightarrow (ab\bar{b}b, \vartheta\overline{\vartheta}\vartheta\overline{\vartheta})$,
\end{itemize}
where $a, b \in \{0,1\}$ and $\vartheta \in \{E,R\}$.
\end{theorem}

\begin{example} \label{ex:norm2}
Let us normalize the directive bi-sequence $\Lambda = ((011)^{\omega}, (EER)^{\omega})$ from Example~\ref{ex:norm}.
According to the procedure from Theorem~\ref{thm:norm}, we transform first prefixes of $\Lambda$. We replace $(0,E)$ with $(01,RE)$ and get $\Lambda_1 = (01(110)^{\omega}, RE(ERE)^{\omega})$. Some prefixes of $\Lambda_1$ are still of a~forbidden form, we replace thus the prefix $(011,REE)$ with $(0110, RERE)$ and get $\Lambda_2 = (0110(101)^{\omega}, RERE(REE)^{\omega})$. Prefixes of $\Lambda_2$ are now correct. It remains to replace from left to right the factors $(101, REE)$ with $(1010, RERE)$. Finally, we obtain $\widetilde{\Lambda} = (0110(1010)^{\omega}, RERE(RERE)^{\omega})=(01(10)^{\omega}, (RE)^{\omega})$, which is already normalized. Let us write down the first prefixes $(\widetilde{w}_n)_{n\geq 0}$ of ${\mathbf u}(\widetilde{\Lambda})$:
\begin{align}
\widetilde{w}_1 =& \;0 \nonumber \\
\widetilde{w}_2 =& \;01 \nonumber \\
\widetilde{w}_3 =& \;0110 \nonumber \\
\widetilde{w}_4 =& \;011001. \nonumber
\end{align}
We can notice that the new sequence $(\widetilde{w}_n)_{n\geq 0}$ now contains the palindromes $0$ and $0110$ that were skipped in Example~\ref{ex:norm}.
\end{example}

\subsection{Normalization over ternary alphabet}\label{sec:norm ternary}
A~normalized directive bi-sequence can be defined similarly as in Definition~\ref{def:norm} over multiliteral alphabet. In contrast to binary alphabet, over ternary alphabet it is still clear that every directive bi-sequence may be normalized, however the algorithm for normalization similar to the one over binary alphabet (Theorem~\ref{thm:norm}) has not been found yet. Fortunately, in order to prove a~sufficient and necessary condition for periodicity of ternary generalized pseudostandard words, even some partial results on normalization over ternary alphabet suffice. Let us introduce these partial results.

We will focus on bi-sequences that contain infinitely many times exactly two distinct antimorphisms including $R$.
\begin{lemma} \label{lem:norm}
Let the directive bi-sequence $(\Delta, \Theta)$ of a~ternary generalized pseudostandard word $\mathbf{u}$ contain as its factor $(abc, \vartheta R R)$, resp., $(abc, R\vartheta \vartheta)$, where $\vartheta \in \{E_0,E_1,E_2\}$ and $a, b, c \in \{0,1,2\}$ satisfy $\vartheta(b) = c$. Denote $w_n = \vartheta(w_n)$, $w_{n+1} = R(w_{n+1})$ and $w_{n+2} = R(w_{n+2})$, resp., $w_n = R(w_n)$, $w_{n+1} = \vartheta(w_{n+1})$ and $w_{n+2} = \vartheta(w_{n+2})$, the corresponding pseudopalindromic prefixes of $\mathbf{u}$. Then between the pseudopalindromic prefixes $w_{n+1}$ and $w_{n+2}$ of the word $\mathbf{u}$ there is a~$\vartheta$-palindromic, resp., palindromic, prefix $w$ of $\mathbf{u}$ followed by the letter $b$.
\end{lemma}
\begin{proof}
Consider the first case, i.e., $(abc, \vartheta R R)$ is a~factor of $(\Delta, \Theta)$. The proof of the second case is left for the reader since it is analogous. Denote $bsb$ the longest palindromic suffix of $w_n b$. Then $w_n = ps$, where $p$ is a~nonempty prefix of $w_n$ because $w_n$ is a~$\vartheta$-palindrome while $s$ is a~palindrome. We have $w_{n+1} = psR(p)$. When constructing $w_{n+2}$, we look for the longest palindromic suffix of $w_{n+1}c$. It is easy to see that such a~suffix is equal to $\vartheta(b)\vartheta(s)c = c\vartheta(s)c$. We can thus write $w_{n+2} = psR(p)\bigl(\vartheta(s)\bigr)^{-1}R\bigl(psR(p)\bigr) = psR(p)\bigl(\vartheta(s)\bigr)^{-1}psR(p)$. Let us rewrite the word
\begin{equation}\label{w}
w=psR(p)\bigl(\vartheta(s)\bigr)^{-1}ps,
\end{equation}
which is a~prefix of $w_{n+2}$. We have
\begin{align}
w = & \;psR(p)\bigl(\vartheta(s)\bigr)^{-1}\vartheta(ps) \nonumber \\[3mm]
= & \;pR(s)R(p)\bigl(\vartheta(s)\bigr)^{-1}\vartheta(s)\vartheta(p) \nonumber \\[3mm]
= & \;pR(ps)\vartheta(p) \nonumber \\[3mm]
= & \;\vartheta(w). \nonumber
\end{align}
We have used the facts: $\vartheta(ps) = ps$, $R(s) = s$ and $\vartheta R=R\vartheta$. Clearly, we have that $|w_{n+1}| < |w| < |w_{n+2}|$, therefore the sequence $(\Delta, \Theta)$ is not normalized and $w$ is the searched $\vartheta$-palindromic prefix of $\mathbf{u}$. Since $w_{n+2} = wR(p)$ and $p$ ends in $b$,  the word $R(p)$ starts in $b$. Consequently, $w$ is followed by $b$.
\end{proof}

\begin{corollary} \label{cor:norm}
Under the assumptions of Lemma~\ref{lem:norm} we have: If the factor $(abc, \vartheta R R)$, resp., $(abc, R \vartheta \vartheta)$, of the directive bi-sequence $(\Delta, \Theta)$ of the word $\mathbf{u}$ is replaced with the factor $(abcb, \vartheta R \vartheta R)$, resp., $(abcb, R\vartheta R\vartheta)$, the same generalized pseudostandard word is obtained.
\end{corollary}
\begin{proof}
Consider again the first case, i.e., $(abc, \vartheta R R)$ is a~factor of the directive bi-sequence, and let for the reader the second one.
Denote $w_n = \vartheta(w_n)$, $w_{n+1} = R(w_{n+1})$ and $w_{n+2} = R(w_{n+2})$ the corresponding pseudopalindromic prefixes of $\mathbf{u}$. Denote further in the same way as in Lemma~\ref{lem:norm} the skipped $\vartheta$-palindrome by $w$. We know that the prefix $w$ is followed by the letter $b$, thus it suffices to show that $(w_{n+1}c)^{\vartheta} = w$ and $(wb)^{R} = w_{n+2}$.

We will start with the first claim. Assume for contradiction that there exists a~$\vartheta$-palindromic prefix $v_1$ such that $(w_{n+1}c)^{\vartheta} =v_1$ and $|v_1| < |w|$. When constructing the $\vartheta$-palindrome $v_1$, we look for the longest $\vartheta$-palindromic suffix $\vartheta(c)vc=bvc$ of the word $w_{n+1}c$. It is not difficult to see that $|w_{n}| \leq |v| <|w_{n+1}|$, where the first relation follows from the fact that $bR(w_n)c$ is a~$\vartheta$-palindromic suffix of $w_{n+1}c$. Moreover, $v \not = R(w_n)$ because otherwise $v_1=w_{n+1}(R(w_n))^{-1}\vartheta(w_{n+1})=w$. The easiest way to check this equality is to notice that both $v_1$ and $w$ are prefixes of $\mathbf u$ and have the same length according to the form of $w$ defined in~(\ref{w}). Since $v$ is a~$\vartheta$-palindromic suffix of the palindrome $w_{n+1}$, its reversal $R(v)$ is a~$\vartheta$-palindromic prefix between $w_n$ and $w_{n+1}$. Since $w_{n+1}$ is the shortest palindrome having the prefix $w_n b$, it has to satisfy at the same time $w_{n+1}=(w_n b)^{R} = (R(v) b)^{R}$. Let us however show that this leads to a~contradiction. Denote $bsb$ the longest palindromic suffix of $w_n b$ and $b\widetilde{s}b$ the longest palindromic suffix of $R(v)b$. Since $w_n$ is a~suffix of $R(v)$, either $|\widetilde{s}| > |w_n|$ or $|\widetilde{s}| = |s|$. If $|\widetilde{s}| > |w_n|$, then $\vartheta(\widetilde{s})$ is a~palindromic prefix shorter than $w_{n+1}$ and longer than $w_n$, which is a~contradiction. If $|\widetilde{s}| = |s|$, then $|(w_n b)^{R}| < |(R(v)b)^{R}|$, which is again a~contradiction, too.

It remains to show that $(wb)^{R} = w_{n+2}$. We have $(w_{n+1}c)^{R} = w_{n+2}$, i.e., $w_{n+2}$ is the shortest palindrome having $w_{n+1}c$ as prefix. Hence the shortest palindrome with the prefix $wb$ has to equal $w_{n+2}$ since $w_{n+1}c$ is a~prefix of $wb$ and $wb$ is a~prefix of $w_{n+2}$.
\end{proof}
\begin{example}
Let us illustrate Lemma~\ref{lem:norm} and Corollary~\ref{cor:norm}. Assume we have already constructed the prefix $w_k = 012$ of a~generalized pseudostandard word. Suppose further that the factor $(120,E_1RR)$ of the directive bi-sequence follows. It is readily seen that the assumptions of Lemma~\ref{lem:norm} are met (in particular we have $E_1(2) = 0$). Let us write down the prefixes $w_{k+1}$, $w_{k+2}$ and $w_{k+3}$.
\begin{align}
w_{k+1} = & \;0121012, \nonumber \\[3mm]
w_{k+2} = & \;01210122101210, \nonumber \\[3mm]
w_{k+3} = & \;0121012210121001210122101210. \nonumber
\end{align}
It is evident that between the prefixes $w_{k+2}$ and $w_{k+3}$ there is an $E_1$-palindrome $$012101221012100121012$$ followed by $2$. Corollary~\ref{cor:norm} moreover states that the generalized pseudostandard word remains the same if we replace the factor $(120,E_1RR)$ of the directive bi-sequence with the factor $(1202,E_1RE_1R)$ -- the reader can check it easily.
\end{example}
\begin{corollary} \label{cor:norm2}
Let the directive bi-sequence $\Lambda = (\Delta, \Theta)$ of a~ternary generalized pseudostandard word $\mathbf{u}$ satisfy: The sequence $\Theta = \vartheta_1\vartheta_2\cdots$ contains infinitely many times exactly two distinct antimorphisms $\vartheta$ and $R$. The sequence $\Delta = \delta_1\delta_2\cdots$ contains infinitely many times two (not necessarily distinct) letters $a, b$ such that $\vartheta(a) =b$. Let further the bi-sequence $\Lambda$ satisfy: There exists $n_0 \in \mathbb{N}$ such that for all $n > n_0$ we have:
either
\begin{align}
\vartheta_n = \vartheta \Rightarrow \delta_{n+1} = a \ \text{and} \ \vartheta_n=R \Rightarrow \delta_{n+1}=b,
\end{align}
or
\begin{align}
\vartheta_n = \vartheta \Rightarrow \delta_{n+1} = b \ \text{and} \ \vartheta_n=R \Rightarrow \delta_{n+1}=a.
\end{align}
Then there exists a~directive bi-sequence $\widetilde{\Lambda} = \bigl(v(ab)^{\omega}, \sigma(R\vartheta)^{\omega}\bigr)$, where $v \in \{0,1,2\}^*$, $\sigma \in \{E_0,E_1,E_2,R\}^*$, such that $\mathbf{u}(\Lambda) = \mathbf{u}(\widetilde{\Lambda})$.
\end{corollary}
\begin{proof}
We can certainly find $m>n_0$ such that the sequence $\Theta$ contains -- starting from the index $m$ -- exactly two antimorphisms $\vartheta$ and $R$ (both of them infinitely many times) and the sequence $\Delta$ contains -- starting from the index $m$ -- only letters $a$ and $b$ (not necessarily distinct). Let us find  $\ell > m$ satisfying $\vartheta_{\ell} = R$ and $\vartheta_{\ell+1} = \vartheta$. Using assumptions of Corollary~\ref{cor:norm2}, we get $\delta_{\ell+1} = b$ and $\delta_{\ell+2} = a$ (we assume without loss of generality that the antimorphism $R$ is followed by the letter $b$). We have thus found a~factor of the directive bi-sequence of the form $(cba, R\vartheta\zeta_1)$ for some $c \in \{a,b\}$ and $\zeta_1 \in \{\vartheta, R\}$. If now $\zeta_1 = \vartheta$, the assumptions of Corollary~\ref{cor:norm} are met, and consequently this factor may be replaced with the factor $(cbab, R\vartheta R\vartheta)$ and we get the same generalized pseudostandard word. If $\zeta_1 = R$, then $\delta_{\ell+3} = b$ and we get again a~factor of the directive bi-sequence of the form $(cbab, R\vartheta R\zeta_2)$ for some $\zeta_2 \in \{\vartheta, R\}$, etc. A formal proof by induction is left as an exercise for the reader. Finally, we set $v := \delta_1\ldots\delta_{\ell+1}$ and $\sigma :=\vartheta_1\ldots\vartheta_{\ell+1}$.
\end{proof}
At this moment, we know that if a~directive bi-sequence satisfies the assumptions of Lemma~\ref{lem:norm}, it is not normalized. The remaining question is whether the new bi-sequence whose existence is guaranteed by Corollary~\ref{cor:norm2} is normalized (at least from a~certain moment on). A~partial answer to this question is provided in the following lemma.
\begin{lemma} \label{lem:norm2}
Let the directive bi-sequence $\Lambda = (\delta_1\delta_2\cdots, \vartheta_1\vartheta_2\cdots)$ of a~generalized pseudostandard word $\mathbf{u}$ be of the form $\Lambda = (v(ab)^{\omega}, \sigma(R\vartheta)^{\omega})$, where $v \in \{0,1,2\}^*$, $\sigma \in \{E_0,E_1,E_2,R\}^*$ and $|v| = |\sigma|$, $\vartheta \in \{E_0,E_1,E_2\}$ and $a, b \in \{0,1,2\}$ such that $\vartheta(a) = b$. Then for all $n > n_0 = |v|$ the sequence $(w_n)_{n>n_0}$ contains all palindromic, resp., $\vartheta$-palindromic, prefixes of length larger than $|w_{n_0}|$ of the word $\mathbf{u}$ followed by the letter $b$, resp., $a$.
\end{lemma}
\begin{proof}
Assume for contradiction that there exists $n > n_0$ such that between the prefixes $w_n = \vartheta(w_n)$ and $w_{n+1} = R(w_{n+1})$ (the converse case is analogous) another pseudopalindromic prefix $w$ occurs:
\begin{itemize}
\item
 Either $w$ is a~palindrome. This is a~contradiction with the fact that $w_{n+1}$ is the shortest palindrome having $w_n a$ as prefix.
 \item Or $w$ is a~$\vartheta$-palindrome followed by $a$ (consider the shortest such $w$). Denote $asa$ the longest palindromic suffix of $w_n a$ and $a\widetilde{s}a$ the longest palindromic suffix of $wa$. Since $w_n$ is a~suffix of $w$, either $|\widetilde{s}| > |w_n|$ or $|\widetilde{s}| = |s|$. If $|\widetilde{s}| > |w_n|$, then $\vartheta(\widetilde{s})$ is a~palindromic prefix shorter than $w_{n+1}$ and longer than $w_n$, which is a~contradiction. If $|\widetilde{s}| = |s|$, then $|(w_na)^{R}| < |(wa)^{R}|$. Since $w_{n+1}$ is the shortest palindrome having the prefix $w_n a$, it has to satisfy at the same time $w_{n+1}=(w_na)^{R} = (wa)^{R}$, and this is again a~contradiction.
     \end{itemize}
\end{proof}

\section{Periodicity of binary generalized pseudostandard words} \label{sec:per}
Our first new result concerning generalized pseudostandard words is a~necessary and sufficient condition for their periodicity over binary alphabet. We thus consider throughout this section binary infinite words.
\begin{theorem}\label{thm:periodicitybinary}
A binary generalized pseudostandard word ${\mathbf u}(\Delta, \Theta)$, where $\Delta=\delta_1\delta_2\ldots \in \{0,1\}^{\mathbb N}$ and $\Theta=\vartheta_1\vartheta_2\ldots \in \{E,R\}^{\mathbb N}$, is periodic if and only if the directive bi-sequence $(\Delta, \Theta)$ satisfies the following condition:
\begin{equation}\label{eq:podm}
(\exists a \in \{0,1\})(\exists \vartheta \in \{E,R\})(\exists n_0 \in \mathbb N)(\forall n>n_0, n \in \mathbb N)(\delta_{n+1}=a \Leftrightarrow \vartheta_n =\vartheta).
\end{equation}
\end{theorem}
\begin{remark}
The condition for periodicity may be rewritten in a slightly less formal way:\\
${\mathbf u}(\Delta, \Theta)$ is periodic if and only if there exists a~bijection $\pi: \{E,R\} \to \{0,1\}$ such that $\pi(\vartheta_n)=\delta_{n+1}$ for all sufficiently large $n$.
\end{remark}

Let us point out that generalized pseudostandard words are either aperiodic or purely periodic -- it follows from the fact that they are recurrent.

In order to prove Theorem~\ref{thm:periodicitybinary} we need the following lemma and remark.
\begin{lemma} \label{lemma:tvar}
Let $(\Delta,\Theta)$ be a normalized directive bi-sequence of a generalized pseudostandard word. Assume $(\Delta,\Theta)$ satisfies condition~\eqref{eq:podm} and both $E$ and $R$ occur in $\Theta$ infinitely many times. Then there exist
$$\nu \in {\{0,1\}}^*,\ \sigma \in {\{E,R\}}^*,\ a, b \in \{0,1\}, \ \vartheta \in \{E,R\}, \ i \in \mathbb N$$ such that
$$\Delta=\nu b a^i (a\bar{a})^\omega \quad \text{and} \quad \Theta=\sigma \vartheta^{i+1}(\overline{\vartheta}\vartheta)^\omega \quad \text{and $|\nu|=|\sigma|$.}$$

\end{lemma}
\begin{proof}
Let us set $\nu= \delta_1 \ldots \delta_{n_0}$ and $\sigma= \vartheta_1 \ldots \vartheta_{n_0}$. Let us further denote $b=\delta_{n_0+1}$ and $\vartheta=\vartheta_{n_0+1}$. Since the directive bi-sequence satisfies condition~\eqref{eq:podm}, the same letter (say $a$) has to follow $\vartheta$.
Since both $E$ and $R$ occur in~$\Theta$ infinitely many times, $\vartheta$ is repeated only finitely many times (say $i+1$ times), i.e., $\vartheta_{n_0+1}=\ldots=\vartheta_{n_0+1+i}=\vartheta$ and $\vartheta_{n_0+2+i}=\overline{\vartheta}$. According to~\eqref{eq:podm} we have $\delta_{n_0+2}=\ldots=\delta_{n_0+2+i}=a$ and $\delta_{n_0+3+i}=\bar{a}$. By Theorem~\ref{thm:norm} a~normalized directive bi-sequence cannot contain the factor $(cd\bar{d},\gamma\bar{\gamma}\bar{\gamma})$ for any $c, d \in \{0,1\}$, $\gamma \in \{E,R\}$. Consequently, $\vartheta_{n_0+3+i}=\vartheta$. Consider now the prefix of $(\Delta, \Theta)$ of the form $\Lambda_k=(\nu b a^i (a\bar{a})^k, \sigma \vartheta^{i+1}(\overline{\vartheta}\vartheta)^k)$. Then again by Theorem~\ref{thm:norm} and using~\eqref{eq:podm}, the prefix of $(\Delta, \Theta)$ of length $|\Lambda_k|+2$ is equal to $\Lambda_{k+1}$.
\end{proof}

The following remark follows easily from Theorem~\ref{thm:norm}.
\begin{remark} \label{rem:norm}
Let a~directive bi-sequence of a generalized pseudostandard word satisfy condition~\eqref{eq:podm}. Then the corresponding normalized bi-sequence satisfies condition~\eqref{eq:podm}.
\end{remark}
%

\begin{proof}[Proof of Theorem~\ref{thm:periodicitybinary}]
\noindent $(\Leftarrow):$
\begin{enumerate}
\item Assume that the sequence $\Theta$ contains both $E$ and $R$ infinitely many times. Let us normalize $\Lambda=(\Delta, \Theta)$ and denote the new directive bi-sequence by $\widetilde{\Lambda}$. By Remark~\ref{rem:norm} the sequence $\widetilde{\Lambda}$ satisfies condition~\eqref{eq:podm}. Applying Lemma~\ref{lemma:tvar} it is possible to write $\widetilde{\Lambda}=(\widetilde{\nu}(a\bar{a})^\omega, \widetilde{\sigma} (\overline{\vartheta}\vartheta)^\omega)$, where $|\widetilde{\nu}|=|\widetilde{\sigma}|$. Without loss of generality suppose that $\widetilde{\sigma}=\widetilde{\theta_1}\vartheta$. (Otherwise we would extend the sequence $\widetilde{\nu}$ and $\widetilde{\sigma}$ by two consecutive members.) Set $n_0 = |\widetilde{\sigma}|$. We will show that for all $n > n_0$ there exists $k \in \mathbb{N}$ such that
either
\begin{equation}\label{eq:tvar1}
w_n = w_{n_0}[(w_{n_0}^{-1}w_{n_0+1})\vartheta\overline{\vartheta}(w_{n_0}^{-1}w_{n_0+1})]^k,
\end{equation}
\begin{center}
or
\end{center}
\begin{equation}\label{eq:tvar2}
w_n = w_{n_0}[(w_{n_0}^{-1}w_{n_0+1})\vartheta\overline{\vartheta}(w_{n_0}^{-1}w_{n_0+1})]^k(w_{n_0}^{-1}w_{n_0+1}),
\end{equation}
where $(w_n)_{n\geq 0}$ is the sequence of prefixes associated with ${\mathbf u}(\widetilde \Lambda)$ (we omit tildes for simplicity).
It follows then directly from these forms that
\begin{equation}\label{eq:period}
w_{n_0+1}\vartheta\overline{\vartheta}(w_{n_0}^{-1}w_{n_0+1})w_{n_0}^{-1}
\end{equation}
is the period of the generalized pseudostandard word ${\mathbf u}(\Delta, \Theta)$.

\vspace{0.2cm}
It is not difficult to show that if $w_n$ is of the form~\eqref{eq:tvar1}, then $w_n = \vartheta(w_n)$. It suffices to take into account that $\vartheta_{n_0}=\vartheta$ and therefore $\vartheta(w_{n_0})=w_{n_0}$ and $\overline{\vartheta}(w_{n_0+1})=w_{n_0+1}$. Similarly, if $w_n$ is of the form~\eqref{eq:tvar2}, then $w_n = \overline{\vartheta}(w_n)$.

\vspace{0.2cm}
Let us proceed by induction: $w_{n_0}$ and $w_{n_0+1}$ are of the form~\eqref{eq:tvar1} or~\eqref{eq:tvar2} -- it suffices to set $k=0$. Let $n>n_0+1$ and assume $w_\ell$ is of the form~\eqref{eq:tvar1} or~\eqref{eq:tvar2} for all $\ell \in \mathbb{N}$, where $n_0+1<\ell \leq n$.

\begin{itemize}
\item Let $w_n$ be of the form~\eqref{eq:tvar1}. Then $w_n=\vartheta(w_n)$ and by condition~\eqref{eq:podm}, we have $\delta_{n+1}=a$. When constructing $w_{n+1}$, we search for the longest $\overline{\vartheta}$-palindromic suffix of $w_n a$. Since $\widetilde{\Lambda}$ is normalized, the longest $\overline{\vartheta}$-palindromic prefix of $w_n$ is $w_{n-1}$, and consequently the longest $\overline{\vartheta}$-palindromic suffix of $w_n$ is $\vartheta(w_{n-1})$. Thanks to the form of $\widetilde{\Lambda}=(\widetilde{\nu}(a\bar{a})^\omega, \widetilde{\sigma} (\overline{\vartheta}\vartheta)^\omega)$ we further know that $w_{n-1}$ is followed by $\bar{a}$. Since $w_n$ is a~$\vartheta$-palindrome, the factor $\vartheta(w_{n-1})$ is preceded by $\vartheta(\bar{a})$. Consequently, $\vartheta(\bar{a})\vartheta(w_{n-1})a$ is a~candidate for the longest $\overline{\vartheta}$-palindromic suffix of $w_n a$. On the one hand, if $\vartheta = R$, this candidate equals $\bar{a}R(w_{n-1})a$, which is an $E$-palindrome. On the other hand, if $\vartheta = E$, then this candidate equals $aE(w_{n-1})a$, which is an $R$-palindrome. Thus it is indeed the longest $\overline{\vartheta}$-palindromic suffix of $w_n a$. Using the induction assumption and since $w_{n-1}$ is a~$\overline{\vartheta}$-palindrome, we have:
    $$w_n = w_{n_0}[(w_{n_0}^{-1}w_{n_0+1})\vartheta\overline{\vartheta}(w_{n_0}^{-1}w_{n_0+1})]^k,$$
$$w_{n-1} = w_{n_0}[(w_{n_0}^{-1}w_{n_0+1})\vartheta\overline{\vartheta}(w_{n_0}^{-1}w_{n_0+1})]^{k-1}(w_{n_0}^{-1}w_{n_0+1}).$$
Consequently, we obtain:
 $$w_{n+1}=(w_n a)^{\overline{\vartheta}}=(w_{n_0+1}\overline{\vartheta}(w_{n_0}^{-1})\vartheta(w_{n-1})a)^{\overline{\vartheta}}=w_n(w_{n_0}^{-1}w_{n_0+1}),$$ which corresponds to the form~\eqref{eq:tvar2}.
\item For $w_n$ of the form~\eqref{eq:tvar2} we proceed analogously.
\end{itemize}
\item Let the directive bi-sequence be of the form $\Lambda=(\nu a^{\omega}, \sigma \vartheta^{\omega})$. (In fact, the generalized pseudostandard word in question is either an $E$-standard or an $R$-standard word with seed as defined in~\cite{BuLuZa}.)
It is known in this case that the word is periodic~\cite{BuLuZa}. Let us rewrite the directive bi-sequence so that $|\nu|=|\sigma|$ and $\sigma=\theta_1 \vartheta$ and let $n_0=|\sigma|$. It can be proven similarly as in the first case that for all $n > n_0$ there exists $k\in \mathbb{N}$ such that
\begin{equation} \label{eq:tvar3}
w_n=w_{n_0}[w_{n_0}^{-1}w_{n_0+1}]^k.
\end{equation}
Therefore the period of the $E$- or $R$-standard word with seed in question is equal to $w_{n_0+1}w_{n_0}^{-1}$.
\end{enumerate}

\noindent $(\Rightarrow):$
We will show that if condition~\eqref{eq:podm} is not satisfied, then the generalized pseudostandard word ${\mathbf u}(\Delta, \Theta)$ is aperiodic. More precisely, we will show that each of its prefixes is a~left special factor. Let us restrict ourselves to the case where $\Theta$ contains $E$ and $R$ infinitely many times. Otherwise, we deal with $E$- or $R$-standard words with seed and the result is known from~\cite{BuLuZa}. The negation of condition~\eqref{eq:podm} reads: for all $a \in \{0,1\}$, for all $\vartheta \in \{E,R\}$, and for all $n_0 \in \mathbb{N}$ there exists $n>n_0$ such that
\begin{equation}
(\delta_{n+1}=a \wedge \vartheta_n = \overline{\vartheta}) \vee (\delta_{n+1}=\bar{a} \wedge \vartheta_n = \vartheta).
\end{equation}
Let $v$ be a~prefix of ${\mathbf u}(\Delta, \Theta)$. Firstly, take $a=0$, $\vartheta=R$, and $n_0 > |v|$, then there exists $n_1>n_0$ such that $(\delta_{n_1+1}=0 \wedge \vartheta_{n_1} = E)\vee (\delta_{n_1+1}=1 \wedge \vartheta_{n_1} = R)$. Secondly, choose $a=1$, $\vartheta=R$, and $n_0 > |v|$, then there exists $n_2>n_0$ such that $(\delta_{n_2+1}=1 \wedge \vartheta_{n_2} = E)\vee (\delta_{n_2+1}=0 \wedge \vartheta_{n_2} = R)$. The following four cases may occur:

\begin{itemize}
\item $\delta_{n_1+1}=0$, $\vartheta_{n_1} = E$ and $\delta_{n_2+1}=1$, $\vartheta_{n_2} = E$:
    In this case, both $w_{n_1}$ and $w_{n_2}$ are $E$-palindromes, thus $E(v)$ is a~suffix of both of them. Since $\delta_{n_1+1}=0$ and $\delta_{n_2+1}=1$, the words $E(v)0$ and $E(v)1$ are factors of ${\mathbf u}(\Delta, \Theta)$. Since the language is closed under $E$, both $1v$ and $0v$ are factors of ${\mathbf u}(\Delta, \Theta)$.
\item $\delta_{n_1+1}=0$, $\vartheta_{n_1} = E$ and $\delta_{n_2+1}=0$, $\vartheta_{n_2} = R$:
    Now, $E(v)$ has the right extension $E(v)0$ and $R(v)$ has the right extension $R(v)0$. Using the fact that the language is closed under $E$ and $R$, one can see that both $1v$ and $0v$ are factors of ${\mathbf u}(\Delta, \Theta)$.
\item $\delta_{n_1+1}=1$, $\vartheta_{n_1} = R$ and $\delta_{n_2+1}=1$, $\vartheta_{n_2} = E$:
    This case is analogous to the previous one.
\item $\delta_{n_1+1}=1$, $\vartheta_{n_1} = R$ and $\delta_{n_2+1}=0$, $\vartheta_{n_2} = R$:
    This case is similar to the first one.
\end{itemize}
\end{proof}

\begin{example}
Consider the directive bi-sequence $\Lambda = ((011)^{\omega}, (EER)^{\omega})$ from Example~\ref{ex:norm}. This sequence satisfies condition~\eqref{eq:podm}. According to Remark~\ref{rem:norm} the normalization of the directive bi-sequence preserves condition~\eqref{eq:podm}. It follows from Example~\ref{ex:norm2} that the normalized form of the directive bi-sequence is $\widetilde{\Lambda} = (01(10)^{\omega}, RE(RE)^{\omega})$. Let us write down the first prefixes $\widetilde{w}_k$ of ${\mathbf u}(\widetilde{\Lambda})$:
\begin{align}
\widetilde{w}_1 =& \;0 \nonumber \\
\widetilde{w}_2 =& \;01 \nonumber \\
\widetilde{w}_3 =& \;0110 \nonumber \\
\widetilde{w}_4 =& \;011001 \nonumber \\
\widetilde{w}_5 =& \;01100110 \nonumber \\
\widetilde{w}_6 =& \;0110011001. \nonumber
\end{align}
In the proof of Theorem~\ref{thm:periodicitybinary}, the formula for the period (not necessarily the smallest one) of ${\mathbf u}(\Lambda)$ was given by~\eqref{eq:period}: $$w_{n_0+1}\vartheta\overline{\vartheta}(w_{n_0}^{-1}w_{n_0+1})w_{n_0}^{-1},$$
where $\vartheta=E$, $n_0=2$, $w_{n_0} = \widetilde{w}_2$, and $w_{n_0+1} = \widetilde{w}_3$. Thus the period equals $0110 = \widetilde{w}_3$. Therefore $\mathbf{u}(\Lambda)={\mathbf u}(\widetilde{\Lambda})= (0110)^{\omega}$.
\end{example}

\section{Periodicity of ternary generalized pseudostandard words}\label{sec:perTernary}
For ternary generalized pseudostandard words, straightforward analogy of the binary case, i.e., of condition~\eqref{eq:podm} from Theorem~\ref{thm:periodicitybinary}, does not work.
\begin{example} Consider the ternary infinite word $\mathbf u=\mathbf{u}((01)^{\omega}, (RE_1)^{\omega})$. It is easy to show that any prefix $p$ of $\mathbf u$ is left special -- both $1p$ and $2p$ are factors of $\mathbf u$, thus $\mathbf u$ is an aperiodic word.
\end{example}
The condition for periodicity gets more complicated.
\begin{theorem}\label{thm:ternary}
Let $\mathbf u=\mathbf u(\Delta, \Theta)$ be a~ternary generalized pseudostandard word over $\{0,1,2\}$. Then $\mathbf u$ is periodic if and only if one of the following conditions is met:
\begin{enumerate}
\item The sequences $\Delta$ and $\Theta$ are eventually constant, i.e., $\Delta=va^{\omega}$ for some $v \in \{0,1,2\}^*$ and $a \in \{0,1,2\}$ and $\Theta=\sigma \vartheta^{\omega}$ for some $\sigma \in \{E_0, E_1, E_2, R\}^*$ and $\vartheta \in \{E_0, E_1, E_2, R\}$.
\item \begin{itemize}
\item $\Theta$ contains exactly two antimorphisms $\vartheta$ and $R$ infinitely many times;
\item $\Delta$ contains two (not necessarily distinct) letters $a$ and $b$ infinitely many times such that $\vartheta(a)=b$;
\item there exists $n_0 \in \mathbb N$ such that for every $n>n_0$ we have either
$$\vartheta_n=\vartheta \Rightarrow \delta_{n+1}=a \ \wedge \ \vartheta_n=R \Rightarrow \delta_{n+1}=b,$$
or
$$\vartheta_n=\vartheta \Rightarrow \delta_{n+1}=b \ \wedge \ \vartheta_n=R \Rightarrow \delta_{n+1}=a.$$
\end{itemize}
\item The normalized directive bi-sequence $(\widetilde{\Delta}, \widetilde{\Theta})$ of $\mathbf u$ satisfies $$(\widetilde{\Delta}, \widetilde{\Theta})=(v(ijk)^{\omega}, \sigma(E_k E_j E_i)^{\omega}),$$ where $v \in \{0,1,2\}^*, \ \sigma \in \{E_0, E_1, E_2, R\}^*, \ |v|=|\sigma|$, and $i,j,k \in \{0,1,2\}$ are mutually different letters.
\end{enumerate}
\end{theorem}
It is worth mentioning that for bi-sequences containing more antimorphisms including $R$, Theorem~\ref{thm:ternary} provides an easy-to-check condition for recognizing periodicity, however for other bi-sequences containing more antimorphisms, it is not too practical since the algorithm for normalization is not known over ternary alphabet.

\begin{example}
Consider $\Lambda=(0(211)^{\omega}, (R E_0 E_0)^{\omega})$. Since $E_0(1)=2$, the second condition of Theorem~\ref{thm:ternary} is satisfied. Let us write down the first few prefixes $w_n$ of $\mathbf u$:
\begin{eqnarray}
w_1 & =& 0 \nonumber \\
w_2 & =& 0210 \nonumber \\
w_3 & =& 0210120210 \nonumber \\
w_4 & =& 0210120210120. \nonumber
\end{eqnarray}
It is left for the reader to show that $\mathbf u=(021012)^{\omega}$.
\end{example}
\begin{example}
Consider $\Lambda=((102)^{\omega}, (E_2 E_0 E_1)^{\omega})$. The third condition of Theorem~\ref{thm:ternary} is satisfied. Let us write down the first few prefixes $w_n$ of $\mathbf u$:
\begin{eqnarray}
w_1 & =& 10 \nonumber \\
w_2 & =& 1002 \nonumber \\
w_3 & =& 100221 \nonumber \\
w_4 & =& 10022110 \nonumber \\
w_5 & =& 1002211002. \nonumber \\
\end{eqnarray}
It is not difficult to see that $\mathbf u=(100221)^{\omega}$.
\end{example}
In order to make the proof of Theorem~\ref{thm:ternary} as comprehensible as possible, let us distinguish several cases according to the number of antimorphisms occurring infinitely many times in the directive bi-sequence. We will consider them in the following three sections. Putting then the results of Sections~\ref{sec:1}, \ref{sec:2} and~\ref{sec:3or4} together, one obtains immediately the proof of Theorem~\ref{thm:ternary}.
\subsection{Directive bi-sequences containing one antimorphism}\label{sec:1}
Let us start with the simplest case, where $\Theta$ contains only one antimorphism infinitely many times.
\begin{lemma}\label{1antimorphism}
Let the directive bi-sequence $\Lambda=(\Delta, \Theta)$ be of the form $\Lambda=(\Delta, \sigma \vartheta^{\omega})$ for some $\sigma \in \{E_0, E_1, E_2, R\}^*$ and $\vartheta \in \{E_0, E_1, E_2, R\}$. Then the word ${\mathbf u}={\mathbf u}(\Delta, \Theta)$ is periodic if and only if $\Delta=va^{\omega}$ for some $v \in \{0,1,2\}^*$ and $a \in \{0,1,2\}$.
\end{lemma}
\begin{proof}
$(\Rightarrow)$: Assume the sequence $\Delta$ contains infinitely many times two distinct letters, say $a$ and $b$. Then the antimorphism $\vartheta$ is followed infinitely many times by both letters $a$ and $b$. This implies that every prefix of $\mathbf u$ is a~left special factor. Therefore $\mathbf u$ is aperiodic.\\
\noindent $(\Leftarrow)$: Let $\Lambda=(va^{\omega}, \sigma \vartheta^{\omega})$. Denote $n_0=\max \{|v|, |\sigma|\}$. Then $w_{n_0+2}=ps\vartheta(p)$, where $\vartheta(a)sa$ is the longest $\vartheta$-palindromic suffix of the word $w_{n_0+1}a$. Let us now construct $w_{n_0+3}$. The longest $\vartheta$-palindromic suffix of $w_{n_0+2}a$ equals $\vartheta(a)w_{n_0+1}a=\vartheta(a)s\vartheta(p)a$. We have used the fact that $w_{n_0+1}=ps=\vartheta(ps)$ and $s=\vartheta(s)$ and that by the form of $\Lambda$, the factor $w_{n_0+1}$ is evidently the longest $\vartheta$-palindromic prefix of $w_{n_0+2}$ followed by $a$. Thus $w_{n_0+3}=ps(\vartheta(p))^2$. Repeating this process we get $\mathbf u=ps(\vartheta(p))^{\omega}$.

\end{proof}
\subsection{Directive bi-sequences containing two antimorphisms}\label{sec:2}
Let us consider directive bi-sequences $(\Delta, \Theta)$, where $\Theta$ contains infinitely many times exactly two distinct antimorphisms. The first lemma holds for any two antimorphisms, while in the sequel we will consider bi-sequences, where one of the antimorphisms equals~$R$.

\begin{lemma}
Let the directive bi-sequence $\Lambda = (\Delta, \Theta)$ satisfy: $\Delta$ contains infinitely many times all letters $0,1,2$ and $\Theta$ contains infinitely many times exactly two antimorphisms $\vartheta_1, \vartheta_2 \in \{E_0, E_1, E_2, R\}$. Then the word $\mathbf{u} = \mathbf{u}(\Delta, \Theta)$ is aperiodic.
\end{lemma}
\begin{proof}
Thanks to the form of $\Lambda$ it is possible to choose two sequences of indices, say $(k_n)^{\infty}_{n=1}$ and $(\ell_n)^{\infty}_{n=1}$, such that for all $n \in \mathbb{N}$ there exist two distinct letters $\delta_1$, $\delta_2$ and an antimorphism $\vartheta$ satisfying $\delta_{k_n} = \delta_1$, $\delta_{\ell_n} = \delta_2$ and $\vartheta_{k_n-1} = \vartheta_{\ell_n-1}=\vartheta$. In other words, the same antimorphism is followed infinitely many times by two distinct letters. This implies that every prefix of $\mathbf{u}$ is left special.
\end{proof}

\begin{proposition} \label{th:per}
Let the directive bi-sequence $\Lambda = (\Delta, \Theta)$ of a~ternary generalized pseudostandard word $\mathbf{u}$ satisfy: The sequence $\Theta = \vartheta_1\vartheta_2\cdots$ contains infinitely many times exactly two distinct antimorphisms $\vartheta$ and $R$. The sequence $\Delta = \delta_1\delta_2\cdots$ contains infinitely many times two (not necessarily distinct) letters $a$ and $b$. Then the word $\mathbf{u} = \mathbf{u}(\Delta,\Theta)$ is periodic if and only if the directive bi-sequence $(\Delta, \Theta)$ satisfies: $\vartheta(a) = b$ and there exists $n_0 \in \mathbb N$ such that for all $n>n_0$ either
\begin{align}\label{eq:podminka2}
\vartheta_n = \vartheta \Rightarrow \delta_{n+1} = a \ \text{and} \ \vartheta_n=R \Rightarrow \delta_{n+1}=b,
\end{align}
or
\begin{align}\label{eq:podminka2b}
\vartheta_n = \vartheta \Rightarrow \delta_{n+1} = b \ \text{and} \ \vartheta_n=R \Rightarrow \delta_{n+1}=a.
\end{align}
\end{proposition}
\begin{proof}
($\Rightarrow$):
\begin{enumerate}
\item If neither~\eqref{eq:podminka2} nor~\eqref{eq:podminka2b} is met, then the letters $a$ and $b$ are necessarily distinct and it happens infinitely many times that the same antimorphism is followed by two distinct letters. This implies that every prefix of $\mathbf u$ is left special.
\item Assume without loss of generality that condition~\eqref{eq:podminka2} is satisfied. If $\vartheta(a) \neq b$, then every prefix of $\mathbf{u}$ is left special: Let $v$ be a~prefix of $\mathbf{u}$. Let $v$ be contained in the prefix $w_k = R(w_k)$, where $k>n_0$. Then by condition~\eqref{eq:podminka2} $w_k b \in \mathcal{L}({\mathbf{u}})$, consequently $R(b)R(w_k) = bw_k \in \mathcal{L}({\mathbf{u}})$. Let us similarly find $\ell$, where $\ell >n_0$, such that $v$ is a~prefix of $w_\ell = \vartheta (w_\ell)$. Then $w_\ell a \in \mathcal{L}({\mathbf{u}})$, and consequently $\vartheta(a)\vartheta(w_\ell) = \vartheta(a)w_\ell \in \mathcal{L}({\mathbf{u}})$. Thanks to the assumption that $\vartheta(a) \neq b$, the prefix $v$ is a~left special factor of $\mathbf{u}$.

\end{enumerate}
($\Leftarrow$): According to Corollary~\ref{cor:norm2} we can assume without loss of generality that $\Lambda = (v(ab)^{\omega}, \sigma(R\vartheta)^{\omega})$, where $v \in \{0,1,2\}^*$, $\sigma \in \{E_0,E_1,E_2,R\}^*$ and $|v| = |\sigma|$. Denote $n_0 = |\sigma |$, i.e., $w_{n_0+1} = R(w_{n_0+1})$ and $w_{n_0+2} = \vartheta(w_{n_0+2})$. Denote further $\vartheta(b)sb=asb$ the longest $\vartheta$-palindromic suffix of the word $w_{n_0+1}b$. We can thus write $w_{n_0+2} = ps\vartheta(p) = p\vartheta(s)\vartheta(p) = p\vartheta(ps)$ for some prefix $p$ of the word $w_{n_0+1}$. Thanks to Lemma~\ref{lem:norm2}, the factor $w_{n_0+1}$ is the longest palindromic prefix of $w_{n_0+2}$ followed by the letter $b$, therefore $\vartheta(b)\vartheta(w_{n_0+1})a=a\vartheta(w_{n_0+1})a=a\vartheta(ps)a$ is the longest palindromic suffix of the word $w_{n_0+2}a$. Consequently, we have $w_{n_0+3} = p\vartheta(ps)R(p) = ps\vartheta(p)R(p)$. Let us show by induction that for all $k \geq 1$ the following holds:
\begin{align}
w_{n_0+2k-1} &= ps(\vartheta(p)R(p))^{k-1},  \label{tvar1}\\[3mm]
w_{n_0+2k} &= ps(\vartheta(p)R(p))^{k-1}\vartheta(p). \label{tvar2}
\end{align}

For $k = 1$ the relations~\eqref{tvar1} and~\eqref{tvar2} hold. Let them hold for some $k \geq 1$. Then $w_{n_0+2k+1} = (w_{n_0+2k}a)^{R}$. We know that $\vartheta(w_{n_0+2k-1})$ is a~palindromic suffix of the word $w_{n_0+2k}$ preceded by the letter $a=\vartheta(b)$ and, thanks to Lemma~\ref{lem:norm2}, this suffix is the longest palindromic suffix of the word $w_{n_0+2k}$ preceded by the letter $a=\vartheta(b)$. It follows:
\begin{align}
w_{n_0+2k+1} =&\;(w_{n_0+2k}a)^{R} \nonumber \\[3mm]
=&\;(ps(\vartheta(p)R(p))^{k-1}\vartheta(p)a)^{R} \nonumber \\[3mm]
=&\;(p\vartheta(w_{n_0+2k-1})a)^{R} \nonumber \\[3mm]
=&\;p\vartheta(w_{n_0+2k-1})R(p) \nonumber \\[3mm]
=&\;ps(\vartheta(p)R(p))^{k}. \nonumber
\end{align}
Similarly, $R(w_{n_0+2k})$ is the longest $\vartheta$-palindromic suffix of the word $w_{n_0+2k+1}$ preceded by the letter $R(a)=a$.
\begin{align}
w_{n_0+2(k+1)} =&\;(w_{n_0+2k+1}b)^{\vartheta} \nonumber \\[3mm]
=&\;(ps(\vartheta(p)R(p))^k b)^{\vartheta} \nonumber \\[3mm]
=&\;(pR(w_{n_0+2k})b)^{\vartheta} \nonumber \\[3mm]
=&\;pR(w_{n_0+2k})\vartheta(p) \nonumber \\[3mm]
=&\;ps(\vartheta(p)R(p))^{k}\vartheta(p). \nonumber
\end{align}
For arbitrary $k \geq 1$ the factor $w_{n_0+2k-1}$ is a~palindrome and is of the form~\eqref{tvar1}, hence $\mathbf{u}$ is periodic with the period  $R(\vartheta(p)R(p)) = pR\vartheta(p)$.
\end{proof}
Let us describe all combinations of antimorphisms and letters from Proposition~\ref{th:per} for which a~periodic word occurs, i.e., for which either condition~\eqref{eq:podminka2} or condition~\eqref{eq:podminka2b} is met and at the same time $\vartheta(a) = b$.
\begin{align}
\vartheta_n &= R \Rightarrow \delta_{n+1} = 0& &\text{and} &\vartheta_n &= E_0 \Rightarrow \delta_{n+1} = 0, \nonumber \\[3mm]
\vartheta_n &= R \Rightarrow \delta_{n+1} = 1& &\text{and} &\vartheta_n &= E_0 \Rightarrow \delta_{n+1} = 2, \nonumber \\[3mm]
\vartheta_n &= R \Rightarrow \delta_{n+1} = 2& &\text{and} &\vartheta_n &= E_0 \Rightarrow \delta_{n+1} = 1,
\nonumber
\end{align}
\begin{align}
\vartheta_n &= R \Rightarrow \delta_{n+1} = 1& &\text{and} &\vartheta_n &= E_1 \Rightarrow \delta_{n+1} = 1, \nonumber \\[3mm]
\vartheta_n &= R \Rightarrow \delta_{n+1} = 0& &\text{and} &\vartheta_n &= E_1 \Rightarrow \delta_{n+1} = 2, \nonumber \\[3mm]
\vartheta_n &= R \Rightarrow \delta_{n+1} = 2& &\text{and} &\vartheta_n &= E_1 \Rightarrow \delta_{n+1} = 0, \nonumber
\end{align}
\begin{align}
\vartheta_n &= R \Rightarrow \delta_{n+1} = 2& &\text{and} &\vartheta_n &= E_2 \Rightarrow \delta_{n+1} = 2, \nonumber \\[3mm]
\vartheta_n &= R \Rightarrow \delta_{n+1} = 0& &\text{and} &\vartheta_n &= E_2 \Rightarrow \delta_{n+1} = 1, \nonumber \\[3mm]
\vartheta_n &= R \Rightarrow \delta_{n+1} = 1& &\text{and} &\vartheta_n &= E_2 \Rightarrow \delta_{n+1} = 0. \nonumber
\end{align}

\begin{example}
Consider the directive bi-sequence $\Lambda = (0(211)^{\omega}, R(E_0E_0R)^{\omega})$. According to Corollary~\ref{cor:norm2} we construct a~new bi-sequence $\widetilde{\Lambda} = (02(12)^{\omega}, RE_0(RE_0)^{\omega})$ such that $\mathbf{u}(\Lambda) = \mathbf{u}(\widetilde{\Lambda})$. Let us write down the first few prefixes of $\mathbf{u}(\widetilde{\Lambda})$:
\begin{align}
\widetilde{w}_1 =&\;0 \nonumber \\[3mm]
\widetilde{w}_2 =&\;0210 \nonumber \\[3mm]
\widetilde{w}_3 =&\;0210120 \nonumber \\[3mm]
\widetilde{w}_4 =&\;0210120210 \nonumber \\[3mm]
\widetilde{w}_5 =&\;0210120210120. \nonumber
\end{align}
Using the notation from the proof of Proposition~\ref{th:per} we have $\widetilde{w}_{n_0+1} = \widetilde{w}_3$, $s = 0120$ and $p = 021$. The period of the word $\mathbf{u}(\widetilde{\Lambda})$ is therefore equal to $pRE_0(p) = 021012$, i.e., $\mathbf{u}(\widetilde{\Lambda}) = (021012)^{\omega}$.
\end{example}

\subsection{Directive bi-sequences containing three or four antimorphisms}\label{sec:3or4}
Let us now treat the remaining bi-sequences. At first, we show that all generalized pseudostandard words having in their directive bi-sequence infinitely many times three or four antimorphisms such that one of them equals $R$ are aperiodic.
\begin{proposition}\label{more_antimorphismR}
Let $\mathbf{u} = \mathbf{u}(\Delta, \Theta)$, where $\Theta$ contains infinitely many times three or four antimorphisms including $R$. Then the word $\mathbf{u}$ is aperiodic.
\end{proposition}
\begin{proof}
Assume without loss of generality that the bi-sequence $(\Delta, \Theta)$ is normalized. (It is easy to see that a~normalized form always exists and the original bi-sequence is its subsequence.) Assume now for contradiction that the word $\mathbf{u}$ is periodic. It holds then:
\begin{enumerate}
\item There exists $n_0$ such that for all $n > n_0$ the following conditions are satisfied:
\begin{itemize}
\item $\vartheta_n = R \Rightarrow \delta_{n+1} = a$,
\item $\vartheta_n = E_i \Rightarrow \delta_{n+1} = b$,
\item $\vartheta_n = E_j \Rightarrow \delta_{n+1} = c$,
\item alternatively $\vartheta_n = E_k \Rightarrow \delta_{n+1} = d$;
\end{itemize}
\item \label{eq:pism} moreover $a = R(a) = E_i(b) = E_j(c) (=E_k(d))$,
\end{enumerate}
where $\Theta = \vartheta_1\vartheta_2\cdots$, $\Delta = \delta_1\delta_2\cdots$, $a,b,c,i,j,k \in \{0,1,2\}$ and $i,j,k$ are mutually different indices. If at least one of the above conditions is not met, then it can be easily shown that every prefix of $\mathbf{u}$ is left special, and the word $\mathbf{u}$ is thus aperiodic.

Consider $\ell \in \{0,1,2\}$ such that $RE_\ell$ occurs infinitely many times in $\Theta$. Denote $(e_na,RE_\ell)$, $n > n_0$, the corresponding factors of the bi-sequence $(\Delta, \Theta)$. Denote $w^{(n)}_m$ the corresponding palindromic prefixes and $w^{(n)}_{m+1}$ the corresponding $E_\ell$-palindromic prefixes and let $e$ denote the letter that follows the $E_\ell$-palindromes $w^{(n)}_{m+1}$. Let us now study the form of $w^{(n)}_{m+2}$. It holds that the longest palindromic suffix of the word $w^{(n)}_{m+1}e$ is thanks to the normalized form of the directive bi-sequence equal to $E_\ell(a)E_\ell(w^{(n)}_{m})e$ since according to condition~\ref{eq:pism} we have $E_\ell(a) = e$. Moreover using again the fact that the directive bi-sequence is normalized, the factor $E_\ell(a)E_\ell(w^{(n)}_{m})e$ is the longest $\vartheta$-palindromic suffix of the word $w^{(n)}_{m+1}e$ for any antimorphism $\vartheta$.

There exists $n_1 > n_0$ such that for all $n > n_1$ the factor $w^{(n)}_{m+2}$ is a~palindrome. Suppose that such an index $n_1$ does not exist. It means that for infinitely many indices $n$, when constructing $w^{(n)}_{m+2}$, we look for a~shorter $\vartheta$-palindromic suffix of the word $w^{(n)}_{m+1}e$ than its palindromic suffix $E_\ell(a)E_\ell(w^{(n)}_{m})e$. If we now extend the longest suffix $E_\ell(a)E_\ell(w^{(n)}_{m})e$ stepwise by one letter to the right and one letter to the left so that the obtained word remains a~palindrome, we get a~palindrome $s^{(n)}$ that is surrounded by distinct letters from right and from left. Since ${\mathcal L}(\mathbf u)$ is closed under the antimorphism $R$, the factor $s^{(n)}$ is a~bispecial factor and the word $\mathbf{u}$ is aperiodic, which is a~contradiction.

Similarly, it is possible to show that there exists $n_2 > n_1$ such that for all $n > n_2$ the factor $w^{(n)}_{m+3}$ is an $E_\ell$-palindrome. 

The above arguments imply that $\Theta = \sigma(RE_\ell)^{\omega}$ for some $\sigma\in \{R,E_0,E_1,E_2\}^*$, which is a~contradiction.
\end{proof}
\begin{example}
Since the construction of bispecials $s^{(n)}$ from the previous proof is quite complicated, let us illustrate it on an example. Consider $\mathbf u=\mathbf u((210)^{\omega}, (E_0 E_1 R)^{\omega})$. This bi-sequence is not normalized, however it serves as a~simple example for illustration of the previous proof. Let us write down the first five prefixes:
\begin{align}
{w}_1 =&\;21 \nonumber \\[3mm]
{w}_2 =&\;2110 \nonumber \\[3mm]
{w}_3 =&\;21100112 \nonumber \\[3mm]
{w}_4 =&\; 21100112200221 \nonumber \\[3mm]
{w}_5 =&\; 211\underline{00112200221100}22001122110. \nonumber
\end{align}
Indeed, $E_0(w_3)=12200221$ is a~palindromic suffix of $w_4$ that may be extended stepwise so that the obtained word $s=00112200221100$ is still a~palindrome. The word $s$ is underlined in the prefix $w_5$. Moreover, we can see that $s$ cannot be extended as a~palindrome any more since the next factor is $1s2$. The language ${\mathcal L}(\mathbf u)$ is closed under $R$, therefore $2s1$ is a~factor of $\mathbf u$, too. Since we work in the proof of Proposition~\ref{more_antimorphismR} with a~normalized bi-sequence, it is guaranteed that $s$ cannot be extended to a~prefix of $\mathbf u$. In such a~case $s$ would be a~skipped palindromic prefix of $\mathbf u$.
\end{example}
In order to treat the remaining directive bi-sequences in the last proposition, the following remark and corollary are needed.
\begin{remark} \label{poz:e}
For $i, j, k \in \{0,1,2\}$ mutually different, we have $E_iE_jE_k = E_j$.
\end{remark}
\begin{corollary}\label{cor:ecka}
Let for some $i \in \{0,1,2\}$ and $v \in \{0,1,2\}^*$ hold $v = E_i(v)$. Let further $j, k \in \{0,1,2\}$ be such that $i, j, k$ are mutually different. Then $E_j(v) = E_kE_j(v)$, i.e., $E_j(v)$ is an $E_k$-palindrome.
\end{corollary}

\begin{proposition}
Let $\mathbf{u} = \mathbf{u}(\Delta, \Theta)$ be a~ternary generalized pseudostandard word, where $\Theta$ contains infinitely many times $E_i, E_j$, or $E_i, E_j, E_k$ with $i, j, k$ mutually distinct. Then $\mathbf{u}$ is periodic if and only if
\begin{equation} \label{eq:tvar}
(\widetilde{\Delta},\widetilde{\Theta}) = (v(ijk)^{\omega}, \sigma(E_iE_kE_j)^{\omega}),
\end{equation}
where $(\widetilde{\Delta},\widetilde{\Theta})$ is the normalized form of $(\Delta, \Theta)$, $v \in \{0,1,2\}^*$, $\sigma \in \{E_0,E_1,E_2,R\}^*$ and $|v|=|\sigma|$.
\end{proposition}
\begin{proof}
$(\Rightarrow)$: Suppose that the normalized bi-sequence is not of the form~\eqref{eq:tvar}. We will show that $\mathbf u=\mathbf u\mathbf(\widetilde{\Delta},\widetilde{\Theta})$ is aperiodic. Assume:
\begin{enumerate}
\item There exists $n_0$ such that for all $n > n_0$ the following conditions are satisfied:
\begin{itemize}
\item $\vartheta_n = E_i \Rightarrow \delta_{n+1} = a$,
\item $\vartheta_n = E_j \Rightarrow \delta_{n+1} = b$,
\item alternatively $\vartheta_n = E_k \Rightarrow \delta_{n+1} = c$;
\end{itemize}
\item moreover $E_i(a) = E_j(b) (=E_k(c))$,
\end{enumerate}
where $\Theta = \vartheta_1\vartheta_2\cdots$, $\Delta = \delta_1\delta_2\cdots$, $a,b,c,i,j,k \in \{0,1,2\}$ and $i,j,k$ are mutually different indices. If at least one of the above conditions is not met, then it can be easily shown that every prefix of $\mathbf{u}$ is left special. Since condition~\eqref{eq:tvar} is not met, two possibilities occur:
\begin{itemize}
\item The bi-sequence $(\widetilde{\Delta},\widetilde{\Theta})$ contains infinitely many times the factor $(def,E_\ell E_s E_s)$ for some $d,e,f,\ell,s \in \{0,1,2\}$ and $\ell, s$ mutually distinct. Denote $w^{(n)}_{m}$, $w^{(n)}_{m+1}$ and $w^{(n)}_{m+2}$ the corresponding $\vartheta$-palindromic prefixes. Thanks to Corollary~\ref{cor:ecka} and thanks to the fact that the bi-sequence $(\widetilde{\Delta},\widetilde{\Theta})$ is normalized, the $E_t$-palindrome $E_t(f)p^{(n)}f = E_s(e)E_s(w^{(n)}_{m})f$, where $t$ is different from $\ell,s$, is the longest $\vartheta$-palindromic suffix of the word $w^{(n)}_{m+1}f$. (We use the equalities: $E_s(f)=E_\ell(e), \ f=E_s E_\ell(e), \ E_t(f)=E_tE_s E_\ell(e)=E_s(e)$.) However, since $w^{(n)}_{m+2}$ is an $E_s$-palindrome, another shorter $\vartheta$-palindromic suffix of the word $w^{(n)}_{m+1}f$ has been chosen when constructing $w^{(n)}_{m+2}$. This implies that if we extend the $E_t$-palindrome $p^{(n)}$ stepwise by one letter from both sides so that the obtained factor is again an $E_t$-palindrome, we get some $E_t$-palindrome $q^{(n)}$ surrounded by two letters $g,h$ satisfying $E_t(g) \neq h$. Since ${\mathcal L}(\mathbf u)$ is closed under $E_t$, the factor $q^{(n)}$ is bispecial and the word $\mathbf{u}$ is thus aperiodic.
\item The bi-sequence $(\widetilde{\Delta},\widetilde{\Theta})$ contains infinitely many times the factor $(def,E_\ell E_s E_\ell)$ for some $d,e,f,\ell, s \in \{0,1,2\}$ and $\ell,s$ mutually different. Analogously as in the previous case, we denote $E_t(f)p^{(n)}f = E_s(e)E_s(w^{(n)}_{m})f$ the longest $\vartheta$-palindromic suffix of the word $w^{(n)}_{m+1}f$. This time, the factor $w^{(n)}_{m+2}$ is an $E_\ell$-palindrome. Consequently, another suitable shorter $\vartheta$-palindromic suffix of the word $w^{(n)}_{m+1}f$ has been chosen when constructing $w^{(n)}_{m+2}$. We can hence find exactly as above bispecial factors, thus the word $\mathbf u$ is aperiodic.
\end{itemize}
To sum up, we get $\widetilde{\Theta}=\sigma(E_i E_j E_k)^{\omega}$ and the form of $\widetilde{\Delta}$ follows from conditions 1 and 2.

\noindent $(\Leftarrow)$: Assume without loss of generality that the normalized directive bi-sequence is of the form $(\widetilde{\Delta},\widetilde{\Theta}) = (v(102)^{\omega}, \sigma(E_0E_1E_2)^{\omega})$. Denote $|v| = |\sigma| = n_0$. Then we have $w_{n_0+1} = ps$, where $2s0$ is the longest $E_1$-palindromic suffix of the word $w_{n_0+1}0$. It follows that $w_{n_0+2} = psE_1(p)$. Thanks to the fact that the bi-sequence is normalized and thanks to Corollary~\ref{cor:ecka}, the longest $E_2$-palindromic suffix of the word $w_{n_0+2}2$ equals $E_1(0)E_1(w_{n_0+1})2 = 2E_1(ps)2 = 2sE_1(p)2$. Therefore $w_{n_0+3} = psE_1(p)E_2(p)$. Similarly, the longest $E_0$-palindromic suffix of the word $w_{n_0+3}1$ is equal to $E_2(2)E_2(w_{n_0+2})1 = 2E_2(psE_1(p))1 = 2sE_1(p)E_2(p)1$. We obtain $w_{n_0+4} = psE_1(p)E_2(p)E_0(p)$.

Repeating the previous steps, we get $\mathbf{u} = ps(E_1(p)E_2(p)E_0(p))^{\omega}$, thus $\mathbf{u}$ is periodic.
\end{proof}

\section{Open problems}

We have provided a~necessary and sufficient condition for periodicity of generalized pseudostandard words over binary and ternary alphabet. More precisely, we have described how the directive bi-sequence of a~generalized pseudostandard word has to look like in order to correspond to a~periodic word. The two cases are surprisingly different -- the ternary case is not at all a~simple generalization of the condition over binary alphabet. Over ternary alphabet, it may happen that in order to decide about periodicity using our result (Theorem~\ref{thm:ternary}), one needs to know the normalized directive bi-sequence. The problem is that we only know that the normalized form of every directive bi-sequence exists, but in contrast to binary alphabet, we have no algorithm for producing the normalized directive bi-sequence from a~given directive bi-sequence over ternary alphabet. Therefore, it is desirable to find such a~normalizing algorithm over ternary or even any alphabet. Section~\ref{sec:norm ternary} may serve as a~hint in such an effort.

Observing results for binary and ternary alphabet, we have the following conjecture for multiliteral alphabet.
\begin{conjecture}[Periodicity of generalized pseudostandard words]\label{multiliteral}
Consider an alphabet $\mathcal A$ with $\#{\mathcal A}=d$ and $G$ the set of all involutory antimorphisms on ${\mathcal A}^*$.
Let $\mathbf u(\Delta, \Theta)$ be a~$d$-ary generalized pseudostandard word, where $\Delta$ is a~sequence of letters from $\mathcal A$ and $\Theta$ is a~sequence of antimorphisms from $G$. Then $\mathbf u$ is periodic if and only if the following conditions are met:
\begin{enumerate}
\item The normalized directive bi-sequence is of the form $$(\widetilde{\Delta}, \widetilde{\Theta})=(v\delta_1\delta_2\delta_3\ldots, \sigma \vartheta_1 \vartheta_2\vartheta_3\ldots),$$ where $|v|=|\sigma|$ and $\vartheta_i(\delta_{i+1})=\vartheta_j(\delta_{j+1})$ for all $i,j \in \mathbb N$.
\item For all $i \in \mathbb N$, if $w$ is a~$\vartheta_i$-palindrome, then $\vartheta_{i+1}(w)$ is a~$\vartheta_{i+2}$-palindrome.
\end{enumerate}
\end{conjecture}
In order to explain that this conjecture is in correspondence with results over binary and ternary alphabet, let us write down the statements for periodicity over binary and ternary alphabet using the normalized directive bi-sequence. Considering Lemma~\ref{lemma:tvar} and Remark~\ref{rem:norm} following Theorem~\ref{thm:periodicitybinary}, we have the next corollary.
\begin{corollary}
Let $(\widetilde{\Delta}, \widetilde{\Theta})$ be the normalized directive bi-sequence of a~binary generalized pseudostandard word $\mathbf u={\mathbf u}(\widetilde{\Delta}, \widetilde{\Theta})$. Then $\mathbf u$ is periodic if and only if one of the following conditions is met:
  \begin{enumerate}
\item $(\widetilde{\Delta}, \widetilde{\Theta})=(va^{\omega}, \sigma\vartheta^{\omega})$ for some $v \in \{0,1\}^*, \sigma \in \{E, R\}^*, |v|=|\sigma|, a \in \{0,1\}$.
 \item $(\widetilde{\Delta}, \widetilde{\Theta})=(v(a\overline{a})^{\omega}, \sigma(RE)^{\omega})$ for some $v \in \{0,1\}^*, \sigma \in \{E, R\}^*, |v|=|\sigma|$, $a \in \{0,1\}$.

\end{enumerate}
\end{corollary}
Using Theorem~\ref{thm:ternary}, Lemma~\ref{lem:norm} and Corollary~\ref{cor:norm2}, we get the following corollary.
\begin{corollary}
Let $(\widetilde{\Delta}, \widetilde{\Theta})$ be the normalized directive bi-sequence of a~ternary generalized pseudostandard word $\mathbf u={\mathbf u}(\widetilde{\Delta}, \widetilde{\Theta})$. Then $\mathbf u$ is periodic if and only if one of the following conditions is met:
\begin{enumerate}
\item $(\widetilde{\Delta}, \widetilde{\Theta})=(va^{\omega}, \sigma\vartheta^{\omega})$ for some $v \in \{0,1,2\}^*$, $\sigma \in \{E_0, E_1, E_2, R\}^*$, $|v|=|\sigma|$, $\vartheta \in \{E_0, E_1, E_2, R\}$ and $a \in \{0,1,2\}$.
 \item $(\widetilde{\Delta}, \widetilde{\Theta})=(v(ab)^{\omega}, \sigma(RE_i)^{\omega})$ for some $v \in \{0,1,2\}^*$, $\sigma \in \{E_0, E_1, E_2, R\}^*$, $|v|=|\sigma|$, $i \in \{0, 1, 2\}$ and $a, b \in \{0,1, 2\}$.
\item $(\widetilde{\Delta}, \widetilde{\Theta})=(v(ijk)^{\omega}, \sigma(E_k E_j E_i)^{\omega})$, where $v \in \{0,1,2\}^*,  \sigma \in \{E_0, E_1, E_2, R\}^*$, $|v|=|\sigma|$ and $i,j,k \in \{0,1,2\}$ are mutually different letters.
\end{enumerate}
\end{corollary}

It is necessary to assume in Conjecture~\ref{multiliteral} that the word $\mathbf u(\Delta, \Theta)$ is $d$-ary if we consider the set $G$ of involutory antimorphisms over a~$d$-ary alphabet as illustrated by the following example.
\begin{example}
Consider ${\mathcal A}=\{0,1,2,3,4\}$ and the following involutory antimorphisms:
\begin{itemize}
\item $E_{014}(0)=0, E_{014}(1)=1, E_{014}(2)=3, E_{014}(3)=2, E_{014}(4)=4,$
\item $E_2(0)=1, E_2(1)=0, E_2(2)=2, E_2(3)=4, E_2(4)=3,$
\item $R(0)=0, R(1)=1, R(2)=2, R(3)=3, R(4)=4.$
\end{itemize}
Then ${\mathbf u}((01)^{\omega}, (E_{014}E_2)^{\omega})={\mathbf u}((01)^{\omega}, (RE_2)^{\omega})=(01)^{\omega}$. Hence, we can see that as soon as we do not work with a~$d$-ary word ($d=5$ in this case), it can happen that the same word is obtained using several normalized directive bi-sequences and the conditions from Conjecture~\ref{multiliteral} are not met despite the evident periodicity of the word $\mathbf u$.
\end{example}
\subsubsection*{Acknowledgements}
This work was supported by the Czech Science Foundation grant GA\v CR 13-03538S and
by the grant of the Grant Agency of the Czech Technical University in Prague SGS14/205/OHK4/3T/14.

\end{document}